\documentclass[12pt]{amsart}
\usepackage[]{amsmath, amsthm, amsfonts, amssymb, epsfig}
\newcommand\C{{\mathbb C}}

\newcommand\dee{\partial}
\newcommand\Om{\Omega}
\newcommand\Obar{\overline{\Omega}}

\numberwithin{equation}{section}

\begin{document}

\title[Improved Riemann Mapping Theorem]
{An improved Riemann Mapping Theorem\\
 and complexity in potential theory}
\author[S.~R.~Bell]
{Steven R. Bell}

\address[]{Mathematics Department, Purdue University, West Lafayette,
IN  47907}
\email{bell@math.purdue.edu}


\begin{abstract}
We discuss applications of an improvement on the Riemann mapping
theorem which replaces the unit disc by another
``double quadrature domain,'' i.e., a domain that
is a quadrature domain with respect to both area and
boundary arc length measure.  Unlike the classic Riemann
Mapping Theorem, the improved theorem allows
the original domain to be finitely connected, and if
the original domain has nice boundary, the biholomorphic
map can be taken to be close to the identity, and consequently,
the double quadrature domain close to the original domain.
We explore some of the parallels between this new theorem
and the classic theorem, and some of the similarities
between the unit disc and the double quadrature domains
that arise here.  The new results shed light on the complexity
of many of the objects of potential theory in multiply
connected domains.
\end{abstract}

\maketitle

\theoremstyle{plain}

\newtheorem {thm}{Theorem}[section]
\newtheorem {lem}[thm]{Lemma}

\hyphenation{bi-hol-o-mor-phic}
\hyphenation{hol-o-mor-phic}

\section{Introduction}
\label{sec1}

The unit disc is the most famous example of a ``double quadrature
domain.''  The average of an analytic function on the disc with
respect to both area measure and with respect to boundary arc length
measure yield the value of the function at the origin when these
averages make sense.  The Riemann Mapping Theorem states that
when $\Om$ is a simply connected domain in the plane that is not equal
to the whole complex plane, a biholomorphic map of $\Om$ to this
famous double quadrature domain exists.

We proved a variant of the Riemann Mapping Theorem in \cite{BGS}
that allows the domain $\Om\ne\C$ to be simply or finitely
connected.  The theorem states that there is a biholomorphic mapping
of such regions onto a one point double quadrature domain, i.e., a
bounded domain such that the average of a holomorphic function with respect
to area measure is a fixed finite linear combination of the function
and its derivatives evaluated at a single point, and such that
the same is true of the average with respect to arc length measure
(with different constants, of course) when these averages make sense.
If the domain is a bounded domain bounded by finitely many $C^\infty$
smooth non-intersecting curves, the double quadrature domain can be
taken to be as $C^\infty$ close to the original as desired, and the
biholomorphic map $C^\infty$ close to the identity.  (For definitions
and a precise statement of the result, see \cite{BGS}.) \ If the boundary
curves are Jordan curves, then a standard argument in conformal
mapping theory shows that we may conformally map the domain to
a close by domain bounded by real analytic curves via a conformal
map that is close to the identity in the sup norm on the closure
of the domain.  If we now apply the $C^\infty$ theorem to the domain
with real analytic boundary and compose, we find that the double
quadrature domain can be taken to be close to the original domain,
and the biholomorphic map close to the identity in the sup norm on
the closure of the domain.

Double quadrature
domains satisfy a long list of desirable properties, many of which
are spelled out in \cite{BGS}.  The purpose of this paper is to add
to that list and derive some consequences from it.  We will show
that double quadrature domains have many properties in common
with the unit disc that allow the objects of potential theory
and complex analysis to be expressed in rather simple terms.
In particular, the action of many classical operators on rational
functions will be seen to be particularly simple.

In \S\ref{secD}, we show that the solution to the Dirichlet
problem with rational data on an area quadrature domain is
real algebraic modulo an explicit finite dimensional subspace.
Since derivatives of the Green's function with respect to
the second variable are solutions to a Dirichlet problem with
rational data, this yields information about the complexity
of the Green's function and the closely related Poisson kernel.
The results yield a method to solve the Dirichlet problem
using only algebra and finite mathematics in place of
the usual analytical methods.

In \S\ref{secM}, we explain how biholomorphic and proper
holomorphic mappings between double quadrature domains
are particularly simple, and in \S\ref{secC}, we show
how to pull back the results of \S\ref{secD} to more
general classes of domains that arise naturally.  We
define two classes, which we name {\it algebraic\/} and
{\it Briot-Bouquet}, that are interesting from this
point of view.

In \S\ref{secO}, we show that many of the classical
operators attached to a double quadrature domain map
rational functions to rational or algebraic functions.
The Szeg\H o projection, for example, maps rational
functions to rational functions as an operator from
$L^2$ of the boundary to itself, and maps rational
functions to algebraic functions as an operator from
$L^2$ of the boundary to the space of holomorphic
functions on the domain.  The Kerzman-Stein operator
maps rational functions to rational functions.  These
results demonstrate that double quadrature domains are very
much like the unit disc in this regard.

In the last section \S\ref{secF}, we explain an
analogy between the techniques of the paper and
classical Fourier analysis.  The analogy reveals
a way to view the main results of the paper on
multiply connected domains as a way of doing Fourier
analysis on multiple curves.

For the history of the study of quadrature domains and the
many applications they have found, see the book \cite{EGKP}
and the article \cite{GS} therein, and Shapiro's classic
text \cite{S}.  Darren Crowdy \cite{C} has shown that double
quadrature domains arise in certain problems in fluid dynamics,
so the present work could find applications in that area.

\section{Basic properties of quadrature domains}

Suppose that $\Om$ is a bounded finitely connected domain in the
plane.  Bj\"orn Gustafsson \cite{G} proved that if $\Om$ is an
area quadrature domain, then $\Om$ must have
piecewise real analytic boundary and the Schwarz function $S(z)$
associated to $\Om$ extends meromorphically to the domain.  (Aharonov
and Shapiro \cite{AS} first proved extendability of the Schwarz function
in the simply connected case.)  Consequently, since
$S(z)=\bar z$ on the boundary, $z$ extends meromorphically
to the double given values $\overline{S(z)}$ on the ``backside,''
and $S(z)$ extends given values $\bar z$ on the backside.
Gustafsson showed that the field of meromorphic functions on
the double is generated by the extensions of $z$ and $S(z)$,
i.e., the extensions form a primitive pair.  This implies that
a meromorphic function on the domain that extends meromorphically to the
double must be a rational combination of $z$ and the Schwarz function.
Since $z$ and $S(z)$ both extend to the double, they are algebraically
dependent.  Therefore $S(z)$ is an algebraic function (as noted
by Gustafsson, and by Aharonov and Shapiro in the simply connected
case).  If $\Om$ has no cusps in the boundary, then Gustafsson
\cite{G} showed that the boundary is given by finitely many
non-intersecting real analytic, real algebraic, curves of a special
form.

If $\Om$ is a boundary arc length quadrature domain, then
Gustafsson showed \cite{G2} that the boundary is real analytic
and the complex unit tangent vector function $T(z)$ must extend
as a meromorphic function on the double.  (Shapiro and Ullemar
\cite{SU} first showed this in the simply connected case.)

If $\Om$ is a double quadrature domain, then all of the properties
above hold and it follows that $T(z)$ must be a rational combination
of $z$ and $S(z)$, and consequently, since $S(z)=\bar z$ on the
boundary, $T(z)$ is a rational function of $z$ and $\bar z$.
It is proved in \cite{BGS} that the Bergman kernel $K(z,w)$ and
the Szeg\H o kernel $S(z,w)$ are both rational combinations of $z$,
$S(z)$, $\bar w$, and $\overline{S(w)}$, and consequently are
rational functions of $z$, $\bar z$, $w$, and $\bar w$ when
restricted to the boundary cross the boundary.  Furthermore, the
complex polynomials belong to both the Bergman span and the
Szeg\H o span associated to the domain.  The Kerzman-Stein
kernel $A(z,w)$ is also a rational function of $z$, $\bar z$, $w$,
and $\bar w$.  (We will show in \S\ref{secO} that the Kerzman-Stein
operator sends rational functions to rational functions on double
quadrature domains.  Many of the ideas used in the work leading
up to the main results of this paper trace back to the papers of
Kerzman and Stein \cite{K-S} and Kerzman and Trummer \cite{K-T}.)

These are the basic properties of quadrature domains that we
shall need in order to proceed.

\section{The Dirichlet problem, Green's function, and Poisson
kernel on quadrature domains}
\label{secD}

Peter Ebenfelt \cite{E} proved that if $\Om$ is a bounded simply
connected area quadrature domain and $R(z,\bar z)$ is a real valued
rational function of $z$ and $\bar z$ with no singularities on
$b\Om\times b\Om$, then the Poisson extension of $R(z,\bar z)$ to
$\Om$ is the real part of a rational function of $z$ and the
Schwarz function $S(z)$ for $\Om$.  Consequently, the Poisson
extension is the real part of an algebraic function.  Ebenfelt's
proof is short and uses a very appealing reflection argument.
In our desperation to generalize Ebenfelt's theorem to bounded
multiply connected quadrature domains, we found two alternate
proofs that we could generalize more readily than the reflection
argument.  Each alternate proof gives new insight into the
extension problem, and we will need both approaches later, so
we now present both alternate proofs of Ebenfelt's theorem in the
simply connected case as a way to launch into the generalizations.
(Besides, you can never have too many proofs of a good theorem.)

Suppose $\Om$ is a bounded simply connected area quadrature domain
with no cusps in the boundary, and suppose $\psi(z)=R(z,\bar z)$
is a real valued rational function of $z$ and $\bar z$ with no
singularities on $b\Om\times b\Om$.  We know that $\Om$ has
$C^\infty$ smooth real analytic boundary and that the Schwarz
function $S(z)$ for $\Om$ extends meromorphically to $\Om$.
Furthermore, the field of meromorphic functions on the double of
$\Om$ is generated by $z$ and $S(z)$.  Since $S(z)=\bar z$ on the
boundary, by writing $R(z,\bar z)=R(z,S(z))$
for $z$ in the boundary, we can see that the only possible type of
singularity for a rational function of $z$ and $\bar z$ on the
boundary would have pole-like behavior at isolated points.  Note
that writing $R(z,\bar z)=R(z,S(z))$ yields an extension of $\psi$
from $b\Om$ to $\Om$ as a meromorphic function and that writing
$R(z,\bar z)=R(\overline{S(z)},\bar z)$ yields an extension of $\psi$
from $b\Om$ to $\Om$ as an antimeromorphic function. The Poisson
extension of $\psi$ to $\Om$ is given as the real part of a
holomorphic function $h$ that extends $C^\infty$-smoothly up to
the boundary.

Note that a meromorphic function that extends smoothly to
the boundary extends meromorphically to the double if and
only if there is an anti-meromorphic function on the domain
that extends smoothly to the boundary with the same boundary
values.

Since $h+\overline{h}=2\psi$ on the boundary, we may write
$$h=-\overline{h} + 2 R(\overline{S(z)},\bar z)$$
on the boundary to see that $h$ extends meromorphically to
the double.  Consequently, $h$ is a rational functions of
$z$ and $S(z)$.  Since, as remarked above, $S(z)$ is algebraic,
it follows that $h$ is algebraic.  Hence, the Poisson extension
is real algebraic.

When we extend this argument to the multiply connected setting,
we will need to use a generalization of the fact proved in \cite{Ba}
(see also \cite{B1}, p.~35), that the Poisson extension of $\psi$
to $\Om$ when $\Om$ is $C^\infty$ smooth and simply connected 
is given by $h+\overline{H}$ where
$$h=\frac{P(S_a\psi)}{S_a}\quad\text{and}\quad
H=\frac{P(L_a\overline{\psi}\,)}{L_a}.$$
Here, $P$ denotes the Szeg\H o projection associated to $\Om$,
$S_a(z)=S(z,a)$ is the Szeg\H o kernel, $L_a(z)=L(z,a)$ is the
Garabedian kernel, and $a$ is a fixed point in $\Om$.
For the basic properties of these objects, see \cite{B1},
pp.~1-35.  Note, in particular, that on a $C^\infty$ smooth bounded
finitely connected domain, $L(z,a)$ is $C^\infty$ smooth
on $\Obar\times\Obar$ minus the diagonal and $S(z,a)$ is $C^\infty$
smooth on $\Obar\times\Obar$ minus the boundary diagonal.  Furthermore,
$L(z,a)$ has a simple pole in $z$ at $z=a$ and is non-vanishing
on $\Obar\times\Obar$ minus the diagonal.  If $\Om$ is simply connected,
then $S(z,a)$ is non-vanishing on $\Obar\times\Obar$ minus
the boundary diagonal.  Since the Szeg\H o projection maps
$C^\infty(b\Om)$ into itself (see \cite{B1}, p.~13), it follows
that $h$ and $H$ are holomorphic functions in $C^\infty(\Obar)$.
In the simply connected case, it is easy to show that $h$ and $H$
extend meromorphically to the double as above, and that is a
major part of what we will do in the multiply connected setting.

The second proof of Ebenfelt's theorem uses the Green's function.
Let $f:\Om\to D_1(0)$ be a Riemann map associated to our simply
connected bounded area quadrature domain.  Aharonov and Shapiro
\cite{AS} showed that $f$ is algebraic.  The Green's function
associated to $\Om$ is given by
$$G(z,w)=-\ln\left|\frac{f(z)-f(w)}{1-f(z)\,\overline{f(w)}}\right|.$$
Since
$$\frac{\dee}{\dee w}
\ln\left|\frac{z-w}{1-z\,\bar w}\right|=
\frac{1-|z|^2}{2(w-z)(1- w \bar z)},$$
it is easy to use the complex chain rule to verify that derivatives of
the form
$$G^{(m)}(z,w):=\frac{\dee^m}{\dee w^m}G(z,w)$$
are rational combinations of $f(z)$ and $\overline{f(z)}$ in the
$z$ variable for fixed $w$.  Note that $G^{(m)}(z,w)$ is harmonic
in $z$, vanishes for $z$ in the boundary, and has a singularity of the
form  $c/(z-w)^m$ in $z$ with $c\ne 0$ at $z=w$.  If $R(z,\bar z)$ is
rational in $z$ and $\bar z$ without singularities on $b\Om$, then
$R(z,S(z))$ extends meromorphically to $\Om$ and is pole free on
the boundary.  It is therefore possible to find constants $c_{jk}$ and
positive integers $m_j$ associated to the poles $w_j$ of the meromorphic
function so that
$$R(z,S(z))-\sum_{j=1}^N\sum_{k=1}^{m_j} c_{jk}G^{(k)}(z,w_j)$$
has removable singularities at the poles $w_j$.  This, therefore, is
the Poisson extension of $R(z,\bar z)$ to the interior.  Since $S(z)$
and $f(z)$ are both algebraic and extend meromorphically to the
double of $\Om$, Ebenfelt's result follows.

We will now generalize these arguments to the multiply connected
setting.  The ideas generalize nicely, but the end results are
considerably more complicated to state.  Although the next
theorem seems long and complicated, the bottom line is that the
solution to the Dirichlet problem with rational boundary data
on a double quadrature domain is real algebraic modulo
an explicit finite dimensional subspace.  (Note that double
quadrature domains are bounded area quadrature domains without
cusps on the boundary because bounded boundary arc length
quadrature domains have real analytic boundaries.)

\begin{thm}
\label{thmA}
Suppose that $\Om$ is a bounded $n$-connected area quadrature
domain with no cusps in the boundary.  Suppose further that
$\psi(z)=R(z,\bar z)$ is a rational function of $z$ and $\bar z$
with no singularities on $b\Om\times b\Om$.  The solution to the
Dirichlet problem with boundary data $\psi$ is equal to 
$$h_0(z)+\overline{H_0(z)} +
\sum_{j=1}^{n-1}c_j\left(h_j(z)+\overline{H_j(z)}+\ln|z-b_j|\right),$$
where $h_0$ and $H_0$ are algebraic meromorphic functions on $\Om$
that are rational combinations of $z$ and the Schwarz function
$S(z)$ for $\Om$.  The $c_j$ are complex constants, and
the points $b_j$ are fixed points in the complement of $\Om$, one
in the interior of each bounded component of $\C-\Om$.
The functions $h_j$ and $H_j$, $j=1,\dots,n-1$ are meromorphic
functions on $\Om$ whose derivatives are algebraic functions
that extend meromorphically to the double of $\Om$, and hence
$h_j'$ and $H_j'$ are rational combinations of $z$ and $S(z)$.
The functions $H_j$, $j=0,1,\dots,n-1$ are holomorphic on $\Om$,
but the $h_j$ may have poles at $n-1$ points in $\Om$ specified
in the proof below.  The constants $c_j$ are determined via the
condition that the poles on $\Om$ should cancel so that
$$h_0 + \sum_{j=1}^{n-1}c_j h_j$$
is holomorphic on $\Om$.
\end{thm}

It should also be noted that all the functions $h_j$ and $H_j$,
$j=0,1,\dots,n-1$, extend $C^\infty$ smoothly up to the boundary
of $\Om$ in Theorem~\ref{thmA}.  The functions $h_0$ and $H_0$
and the constants $c_j$ depend on the boundary data, but the
$h_j$ and $H_j$ for $j=1,2,\dots,n-1$ do not.  Since $z$ and $S(z)$
extend to form a primitive pair for the double, the functions that
extend to the double in Theorem~\ref{thmA} can be expressed as
$$\sum_{k=0}^{N-1} P_k(z)S(z)^k,$$
where $P_k$ are rational functions in $z$ and $N$ is the number
of poles of $S(z)$ in $\Om$ counted with multiplicities (see Farkas
and Kra \cite{F-K}, p.~249).  It is interesting to note that $h_0$
and $H_0$ are rational functions of $z$ and $\bar z$ when restricted
to the boundary that extend algebraically to the interior of $\Om$.

Theorem~\ref{thmA} can be interpreted to mean that the solution
to the Dirichlet problem with rational boundary data on a bounded
are quadrature domain without cusps in the boundary is a real
algebraic function $h_0+\overline{H_0}$ modulo an $n-1$ dimensional
subspace spanned by the functions
$$h_j+\overline{H_j}+\ln|z-b_j|,$$
which also have a rather simple structure.  In particular, since
the derivatives of $h_j$ and $H_j$ are rational functions of $z$
and $S(z)$ where $S(z)$ is algebraic, these functions are abelian
integrals in the classical sense.  We will show that there
are other interesting domain functions that can serve as the basis
for the $n-1$ dimensional subspace in Theorem~\ref{thmA} as we
proceed.

The starting point for proving Theorem~\ref{thmA} is the following
general theorem.  Note that, on an area quadrature domain, rational
functions of $z$ and $\bar z$ without singularities on the boundary
are precisely the functions on the boundary that are restrictions
to the boundary of meromorphic functions on the double without poles
on the boundary.  Indeed, as we have seen, since $S(z)=\bar z$ on
the boundary, we have $R(z,\bar z)=R(z,S(z))$ on the boundary, which
extends meromorphically to the double if $R$ is rational, and which
must have no poles on the boundary curves if $R(z,\bar z)$ has no
singularities.  Conversely, any meromorphic function on the double
can be expressed as a rational combination of $z$ and $S(z)$, and
hence, when restricted to the boundary, as a rational function of
$z$ and $\bar z$.

\begin{thm}
\label{thmB}
Suppose that $\Om$ is a bounded $n$-connected domain bounded
by $n$ non-intersecting $C^\infty$ smooth curves, and suppose
$\psi$ is a $C^\infty$ function on the boundary of $\Om$ that
is the restriction to the boundary of a meromorphic function on
the double of $\Om$.  Then the solution to the Dirichlet problem
on $\Om$ with boundary data $\psi$ is given by
$$h_0(z)+\overline{H_0(z)} +
\sum_{j=1}^{n-1}c_j\left(h_j(z)+\overline{H_j(z)}+\ln|z-b_j|\right),$$
where $h_0$ and $H_0$ are meromorphic functions on $\Om$ that extend
meromorphically to the double of $\Om$.
The functions $h_j$, $j=1,\dots,n-1$, are meromorphic functions
on $\Om$ such that
$$\left(h_j' - \frac{1}{2}\frac{1}{z-b_j}\right)\,dz$$
extends to the double as a meromorphic $1$-form.
The functions $H_j$,
$j=1,\dots,n-1$, are holomorphic functions on $\Om$ such that
$$\left(H_j' - \frac{1}{2}\frac{1}{z-b_j}\right)\,dz$$
extends to the double as a meromorphic $1$-form.
The $H_j$, $j=0,1,\dots,n-1$ are holomorphic on $\Om$,
but the $h_j$ may have poles at $n-1$ points in $\Om$ specified
in the proof below.  The $c_j$ are complex constants, and the
points $b_j$ are fixed points in the complement of $\Om$, one
in the interior of each bounded component of $\C-\Om$.  The constants
$c_j$ are determined via the condition that the poles on $\Om$ should
cancel so that $$h_0 + \sum_{j=1}^{n-1}c_j h_j$$
is holomorphic on $\Om$.
\end{thm}

We will explain how Theorem~\ref{thmA} follows from
Theorem~\ref{thmB} after we prove Theorem~\ref{thmB}.

\begin{proof}[Proof of Theorem~\ref{thmB}]
Let $a$ be a point in $\Om$ such that the $n-1$ zeroes of
$S(z,a)$ are distinct and simple.  (That such an $a$ exists is proved in
\cite{Ba}, or see \cite{B1}, p.~106.)  Let $b_1, b_2,\dots,b_{n-1}$
be points, one from the interior of each of the bounded components
of $\C-\Om$.  We will solve the Dirichlet problem with boundary data
$\psi$ by the method described on p.~1367 of \cite{B} (or better,
see pp.~53-55 of \cite{B1}).  There,
it is shown that there are constants $c_1,c_2,\dots,c_{n-1}$ such
that the Poisson extension of $\psi$ is given by
$$h(z)+\overline{H(z)} + \sum_{j=1}^{n-1}c_j\ln|z-b_j|,$$
where, writing $\phi=\psi- \sum_{j=1}^{n-1}c_j\ln|z-b_j|$, we have
$$h=\frac{P(S_a\phi)}{S_a}\quad\text{and}\quad
H=\frac{P(L_a\overline{\phi}\,)}{L_a}.$$
Here, both $h$ and $H$ are holomorphic on $\Om$.
($h$ looks like it might have possible simple poles at the $n-1$ simple
zeroes of $S_a$, but the method ensures that the numerator vanishes
at the zeroes of $S_a(z)$.)
Both $h$ and $H$ extend $C^\infty$ smoothly to the boundary.
Let ${\mathcal L}_j(z)=\ln|z-b_j|$.
We now consider $h$ as a linear combination of
$$h_0=\frac{P(S_a\psi)}{S_a}$$
and the functions
$$h_j=-\frac{P(S_a{\mathcal L}_j)}{S_a}.$$
We now apply Theorem~6.1 of \cite{B} (see also Theorem~14.1 of
\cite{B1} on p.~53) which states that a $C^\infty$ smooth function
$u$ on the boundary can be expressed as $f+\overline{F}$ on the
boundary where
$$f=\frac{P(S_au)}{S_a}$$
is a meromorphic function that extends $C^\infty$-smoothly to the
boundary and
$$F=\frac{P(L_a\bar u)}{L_a}$$
is a holomorphic function that extends $C^\infty$-smoothly to the
boundary.  Hence,
$$h_0=\frac{P(S_a\psi)}{S_a}=
\psi - \frac{\overline{P(L_a\overline{\psi}\,)}}{\overline{L_a}}.$$
Since $\psi$ extends to the double as a meromorphic function without
poles on the boundary curves of $\Om$, it follows that $\psi(z)=\overline{G(z)}$
on the boundary, where $G$ is a meromorphic function on $\Om$ that extends
smoothly to the boundary.  Hence, it follows that $h_0$ has the same boundary
values as a function that extends to $\Om$ as an antimeromorphic function.
Hence $h_0$ extends to the double.

We will now show that the functions $h_j$ have special extension properties, too.
We may reason as above to see that
$$h_j=-\frac{P(S_a{\mathcal L}_j)}{S_a}=-{\mathcal L_j} +
\frac{ \overline{P(L_a \overline{{\mathcal L}_j }\,)}}{\overline{L_a}}.$$
Hence,
$$h_j(z)=-\ln|z-b_j| + \overline{g(z)}$$
on the boundary, where $g$ is a holomorphic function on $\Om$ that extends smoothly
to the boundary.  Notice that if $u$ is a function defined on a neighborhood of the
boundary of $\Om$ and $z(t)$ parametrizes the boundary in the standard sense, then
$$\frac{d}{dt}u(z(t))=\frac{\dee u}{\dee z}\,z'(t) + 
\frac{\dee u}{\dee \bar z}\,\overline{z'(t)}.$$
Apply this idea to both sides of the last formula for $h_j$, divide by $|z'(t)|$,
and note that $z'(t)/|z'(t)|$ is equal to the complex unit tangent
vector function $T(z)$ at $z=z(t)$ to obtain
$$h_j'(z)T(z)=-\frac{1}{2}\frac{T(z)}{z-b_j}
-\frac{1}{2}\frac{\overline{T(z)}}{\bar z-\bar b_j}
+ \overline{g'(z)}\,\overline{T(z)}$$
for $z$ in the boundary.  This shows that
\begin{equation}
\label{A1}
\left(h_j'(z)+\frac{1}{2}\frac{1}{z-b_j}\right)T(z) =
\left(\overline{g'(z)} -\frac{1}{2}\frac{1}{\bar z-\bar b_j}\right)
\overline{T(z)}
\end{equation}
on the boundary.  It follows that
$$\left(h_j'(z)+\frac{1}{2}\frac{1}{z-b_j}\right)\ dz$$
extends to the double of $\Om$ as a meromorphic $1$-form.

Very similar reasoning can be used to handle $H$.  Indeed,
consider $H$ as a linear combination of
$$H_0=\frac{P(L_a\overline{\psi},)}{L_a}$$
and the functions
$$H_j=-\frac{P(L_a{\mathcal L}_j)}{L_a}.$$
Since $L_a$ has a simple pole at $a$ and no zeroes on $\Obar-\{a\}$,
these functions are holomorphic on $\Om$.  Now, we may apply Theorem~6.1
of \cite{B} as we did with $h_0$ to see that
$$H_0=\frac{P(L_a\overline{\psi})}{L_a}=
\overline{\psi} - \frac{\overline{P(S_a\psi)}}{\overline{S_a}}.$$
Since $\psi$ extends to the double as a meromorphic function without
poles on the boundary curves of $\Om$, it also follows that $\psi(z)=G(z)$
on the boundary, where $G$ is a meromorphic function on $\Om$ that extends
smoothly to the boundary.  Hence, it follows that $H_0$ extends 
meromorphically to the double.  Next, we apply the same idea to $H_j$
to see that
$$H_j=-\frac{P(L_a{\mathcal L}_j)}{L_a}=-{\mathcal L}_j +
\frac{ \overline{P(S_a{\mathcal L}_j)}}{\overline{S_a}}.$$
Hence,
$$H_j(z)=-\ln|z-b_j| + \overline{g(z)}$$
where $g$ is a meromorphic function on $\Om$ that extends smoothly to the boundary.
It follows that
$$H_j'(z)T(z)=-\frac{1}{2}\frac{T(z)}{z-b_j}
-\frac{1}{2}\frac{\overline{T(z)}}{\bar z-\bar b_j}
+ \overline{g'(z)}\,\overline{T(z)}$$
for $z$ in the boundary.  This shows that
\begin{equation}
\label{A2}
\left(H_j'(z)+\frac{1}{2}\frac{1}{z-b_j}\right)T(z)=
\left(\overline{g'(z)} -\frac{1}{2}\frac{1}{\bar z-\bar b_j}\right)
\overline{T(z)}
\end{equation}
on the boundary.  It follows that
$$\left(H_j'(z)+\frac{1}{2}\frac{1}{z-b_j}\right)\ dz$$
extends to the double of $\Om$ as a meromorphic $1$-form.
Finally, the last sentence in the theorem about the cancellation of
the poles follows from the proof of the formula for the solution
of the Dirichlet problem given in \cite{B1}.
\end{proof}

Conditions like equations~(\ref{A1}) and~(\ref{A2})
that appear in the proof of Theorem~\ref{thmB} yield that
$$h_j'(z)+\frac{1}{2}\frac{1}{z-b_j}\quad\text{and}\quad
H_j'(z)+\frac{1}{2}\frac{1}{z-b_j}$$
belong to the Class~$\mathcal A$ of \cite{B4}, and are therefore given
by finite linear combinations of functions from a list of simple
domain functions, as shown in \cite{B4}.

\begin{proof}[Proof of Theorem~\ref{thmA}]
Since $\Om$ is a bounded area quadrature domain, we know that the Schwarz
function $S(z)$ for $\Om$ extends meromorphically to $\Om$ and $S(z)=\bar z$
on the boundary.  We also know that the field of meromorphic functions on
the double is generated by the extensions of $z$ and $S(z)$.   It
was shown in \cite{B5} (Theorem~1.5) that $T(z)^2$ extends to the
double as a meromorphic function on such domains.  (If $\Om$
is also a boundary arc length quadrature domain, then $T(z)$ itself
extends to the double as a meromorphic function.)
These simple facts are all that is needed to deduce Theorem~\ref{thmA}
from Theorem~\ref{thmB}.  The conclusions about $h_0$ and $H_0$
follow immediately.  Since $T(z)^2$ extends meromorphically to the double,
it is given by the boundary values of a meromorphic function $\tau$ on
$\Om$ that extends smoothly to the boundary.
Now, noting that $1/T(z)=\overline{T(z)}$, equations~(\ref{A1})
and~(\ref{A2}) can be further manipulated to yield the conclusions in
Theorem~\ref{thmA} about $h_j$ and $H_j$.  Indeed, equation~(\ref{A1})
yields that
$$h_j'(z)=-\frac{1}{2}\frac{1}{\overline{S(z)}-b_j}+
\left(\overline{g'(z)} -\frac{1}{2}\frac{1}{\bar z-\bar b_j}\right)
\overline{\tau(z)}$$
on the boundary, and this shows that $h_j'$ extends meromorphically to
the double.  Consequently, $h_j'$ is a rational combination of $z$ and $S(z)$.
Similarly, $H_j'$ can be shown to satisfy the same conditions.
\end{proof}

Another nice consequence of Theorem~\ref{thmA} is the following
result about the solution to the Dirichlet problem with
rational data on an area quadrature domain without cusps on
the boundary.  The Bergman span associated to a domain is the complex
linear span of functions of $z$ of the form $K^{(m)}(z,w)$
where $w$ are points in the domain, $K^{(0)}(z,w)=K(z,w)$,
and $K^{(m)}(z,w)=(\dee/\dee\bar w)^m K(z,w)$ are
derivatives of the Bergman kernel in the second variable.

\begin{thm}
\label{thm2a}
Suppose that $\Om$ is a bounded $n$-connected area quadrature
domain without cusps on the boundary, and suppose
$\psi(z)=R(z,\bar z)$ is a rational function of $z$ and
$\bar z$ with no singularities on $b\Om\times b\Om$.
The solution to the Dirichlet problem with boundary data
$\psi$ is equal to
$$u=h+\overline{H} +
\sum_{j=1}^{n-1}c_j\ln|z-b_j|,$$
where $h$ and $H$ are holomorphic functions on $\Om$ such that
all their derivatives are algebraic functions that are rational
combinations of $z$ and the Schwarz function $S(z)$ for $\Om$.
Furthermore, $h''$ and $H''$ and all higher
order derivatives are in the Bergman span.  The functions
$\frac{\dee u}{\dee z}$ and $\frac{\dee u}{\dee\bar z}$
are holomorphic and antiholomorphic functions on $\Om$,
respectfully, that extend meromorphically and antimeromorphically
to the double.  All the real partial derivatives of
$u$ are real algebraic functions.  The points $b_j$ are fixed
points in the complement of $\Om$, one in the interior of each
bounded component of $\C-\Om$.
\end{thm}

\begin{proof}[Proof of Theorem~\ref{thm2a}]
To prove Theorem~\ref{thm2a}, we need Theorem~1.4 in
\cite{B5}, which states that, if $g$ is a holomorphic
function that extends meromorphically to the double of a
bounded area quadrature domain with no cusps on the boundary,
then $g'$ extends meromorphically to the double.  Hence,
if we express the solution $u$ to the Dirichlet problem as we
did in the proof of Theorem~\ref{thmB} via
$$u=h+\overline{H}+\sum_{j=1}^{n-1}c_j\ln|z-b_j|,$$
then
$$\frac{\dee u}{\dee z} = h' +\sum_{j=1}^{n-1}c_j\frac{1}{2(z-b_j)},$$
which extends to the double because Theorem~\ref{thmB} plus
Theorem~1.4 in \cite{B5} yield that
$$h'=h_0'+\sum_{j=1}^{n-1}c_jh_j'$$
extends to the double (and $1/(z-b_j)$ extends to the double
because $z$ does).  Similarly,
$H'$ extends meromorphically to the double, and
$$\frac{\dee u}{\dee\bar z} = \overline{H'}
+\sum_{j=1}^{n-1}c_j\frac{1}{2(\bar z-\bar b_j)}$$
extends to the double as an antimeromorphic function.

Note that once it is known that a function $f$ extends
meromorphically to the double, it is seen to be a rational
function of $z$ and $S(z)$.  Since $S(z)$ extends to the
double, so does $S'(z)$.  Consequently, $S'(z)$ is a
rational function of $z$ and $S(z)$, and it follows
inductively that all the derivatives of $f$ extend to
the double as rational combinations of $z$ and $S(z)$.

We now combine Theorem~\ref{thm2a} with Theorem~1.4 in
\cite{B5} and Lemma~4.1 in \cite{B4a} to complete the proof.
Lemma~4.1 in \cite{B4a} states that, if $g$ is a holomorphic
function that extends meromorphically to the double of a
bounded smoothly bounded finitely connected domain with no
poles on the boundary, then $g'$ is in the Bergman span of
the domain.  Hence, if we differentiate a second time,
we see that $h''$ and $H''$ belong to the Bergman span, and
we may repeat this process to see that all higher order
derivatives also belong to the Bergman span.
\end{proof}

The arguments leading to the conclusions of Theorems~\ref{thmB} and~\ref{thmA}
can be repeated using the harmonic measure functions $\omega_j(z)$ in place
of $\ln|z-b_j|$ to obtain similar results.  If $\gamma_j$ is one of the
$n-1$ boundary curves of $\Om$ bounding one of the bounded components of
$\C-\Om$, then $\omega_j$ is the solution to the Dirichlet problem with
boundary values equal to one on $\gamma_j$ and equal to zero on the other
boundary curves.  The Dirichlet problem can be solved in a manner very
similar to the Szeg\H o projection method used above using the $\omega_j$
in place of $\ln|z-b_j|$.  Indeed, this method was developed in \cite{B}
and \cite{Ba} and is described on pp.~89-91 of \cite{B1}.  The proofs of
Theorems~\ref{thmB} and~\ref{thmA} can be repeated line for line and at
the point where equations~(\ref{A1}) and~(\ref{A2}) come into play, the
following argument is invoked.  The classical functions $F_j'$ are
defined via $F_j'(z) = 2(\dee/\dee z)\omega_j(z)$.  (They are
holomorphic on $\Om$, but they are not the derivative of a holomorphic
function on $\Om$ in spite of that prime in the notation.  They are
locally the derivative of a holomorphic function with real part
equal to $\omega_j$.)  In the proof, where the formula
$$h_j(z)=-\omega_j(z) + \overline{g(z)}$$
is obtained, differentiate along the boundary as before to obtain
$$h_j'(z)T(z)=-\frac{1}{2}F_j'(z)T(z)
-\frac{1}{2}\overline{F_j'(z)T(z)}
+ \overline{g'(z)}\,\overline{T(z)}.$$
This shows that $(h_j'+\frac{1}{2}F_j')dz$ extends as a meromorphic $1$-form
to the double.  If $\Om$ is an area quadrature domain, then the
functions $F_j'$ extend meromorphically to the double of $\Om$ (see
Theorem~1.1 of \cite{B5}).  The rest of the proof follows smoothly
after this point.  It is worth stating the resulting theorem for
area quadrature domains here.

\begin{thm}
\label{thm3}
Suppose that $\Om$ is a bounded $n$-connected area quadrature
domain without cusps in the boundary, and suppose
$\psi(z)=R(z,\bar z)$ is a rational function of $z$ and
$\bar z$ with no singularities on $b\Om\times b\Om$.
The solution to the Dirichlet problem with boundary data
$\psi$ is equal to 
$$h_0+\overline{H_0} +
\sum_{j=1}^{n-1}c_j\left(h_j+\overline{H_j}+\omega_j\right),$$
where $h_0$ and $H_0$ are algebraic meromorphic functions on $\Om$
that are rational combinations of $z$ and the Schwarz function
$S(z)$ for $\Om$.  The $c_j$ are complex constants.
The functions $h_j$ and $H_j$, $j=1,\dots,n-1$ are meromorphic
functions on $\Om$ whose derivatives are algebraic functions
that extend meromorphically to the double of $\Om$, and hence
$h_j'$ and $H_j'$ are rational combinations of $z$ and $S(z)$.
The functions $H_j$, $j=0,1,\dots,n-1$ are holomorphic on $\Om$,
but the $h_j$ may have poles at $n-1$ points in $\Om$ specified
in the proof above.  The constants $c_j$ are determined via the
condition that the poles on $\Om$ should cancel so that
$$h_0 + \sum_{j=1}^{n-1}c_j h_j$$
is holomorphic on $\Om$.
\end{thm}

The theorem analogous to Theorem~\ref{thm3} corresponding to
Theorem~\ref{thmB} with $\omega_j$ in place of $\ln|z-b_j|$ also
holds.

We have generalized the first alternate proof of Ebenfelt's theorem
to the multiply connected setting.  Next, we generalize the second
alternate proof.  In \cite{B1a}, \cite{B2}, and \cite{B4}, the derivative
$G^{(1)}(z,w)=(\dee/\dee w)G(z,w)$ is expressed in terms of
simpler domain functions of one variable in various ways.  The
virtue of this second proof allows us to apply these results
to obtain information about the form of the solution to the Dirichlet
problem with rational data in double quadrature domains.

Suppose that $\Om$ is a bounded $n$-connected double quadrature
domain.  Theorem~1.1 of \cite{B2} states that
there are points $A_1$, $A_2$, $A_3$  and $w_k$, $k=1,2,\dots,n-1$
in $\Om$ such that
\begin{equation}
\label{green1}
G^{(1)}(z,w)=r_0(z,w)+\sum_{k=1}^{n-1} \rho_k(z)r_k(w)
\end{equation}
where $r_0(z,w)$ is a rational combination of
$S(w,A_j)$, $S(z,A_j)$, and the conjugate of $S(z,A_j)$
for $j=1,2,3$, $r_k(w)$ is a rational combination of
$S(w,A_j)$ for $j=1,2,3$, and $\rho_k$ is given by
$$\rho_k(z)=G^{(1)}(z,w_k)-\frac{S(z,w_k)L(z,w_k)}{S(w_k,w_k)}.$$
On a double quadrature domain, the Szeg\H o kernel $S(z,b)$ extends
meromorphically to the double in $z$ for each fixed $b\in\Om$.
The same holds for the Garabedian kernel $L(z,b)$.  The functions
$G^{(1)}(z,w_k)$ are solutions to the Dirichlet problem with
boundary data $(1/2)(z-w_k)^{-1}$.  Consequently, we may apply
Theorem~\ref{thmA} to express $G^{(1)}(z,w_k)$ in elementary
terms.  When we collect everything together, we find that
\begin{equation}
\label{green2}
G^{(1)}(z,w)=R_0(z,w)+\sum_{k=1}^{n-1} \sigma_k(z)R_k(w)
\end{equation}
where $R_0(z,w)$ is a rational combination of
$w$, $S(w)$, $z$, $\bar z$, $S(z)$, and $\overline{S(z)}$,
$R_k(w)$ is a rational combination of
$w$ and $S(w)$, and $\sigma_k(z)$ is equal to the function
$(h_k(z)+\overline{H_k(z)}+\ln|z-b_k|)$ appearing in Theorem~\ref{thmA}.
Higher order derivatives $G^{(m)}(z,w)$ have the same form
since $S'(w)$ also extends to the double, and is therefore a rational
function of $w$ and $S(w)$.  Hence, we could use the idea of the
second proof of Ebenfelt's theorem to first extend $R(z,\bar z)$
as a meromorphic function, and then to subtract off a linear
combination of terms $G^{(m)}(z,w_j)$ to remove the poles to
obtain the solution of the Dirichlet problem with boundary
data $R(z,\bar z)$.  Although formula~(\ref{green2}) is new,
the form of the solution to the Dirichlet problem obtained is
very similar to that given by Theorem~\ref{thmA}.

We remark here that if we use Theorem~8.1 of \cite{B1a} to express
the derivatives of the Green's function instead, we obtain the
following theorem, which does seem new.

\begin{thm}
\label{thm3aa}
If $\Om$ is a bounded double quadrature domain, then the
solution to the Dirichlet problem with rational boundary
data is equal to
$$\Phi + \sum_{k=1}^{n-1}c_k (\lambda_k-\omega_k)$$
where  $\Phi$ is a function in $C^\infty(\Obar)$
that is a rational combination of $z$, the Schwarz function
for $\Om$, and the conjugates of these two functions.  Here,
$$\lambda_k(z)=\frac{1}{S(z,z)}\int_{w\in\gamma_k}|S(z,w)|^2 ds$$
are the non-harmonic measure functions studied in \cite{B1a}.
\end{thm}

The Poisson kernel $p(z,w)$ associated to a bounded smoothly bounded
domain $\Om$ is related to the classical Green's function via
$$p(z,w)=-\frac{i}{\pi}\frac{\dee G}{\dee w}(z,w)T(w)$$
where $z\in\Om$ and $w\in b\Om$.  On a double quadrature domain,
since $T(w)$ extends meromorphically to the double, $T(w)$ is therefore
a rational function of $w$ and $\bar w=S(w)$, and it follows from
equation~(\ref{green2}) that the Poisson kernel associated to
the domain is equal to an expression analogous to the right hand
side of equation~(\ref{green2}) as follows.

\begin{thm}
\label{thm3a}
The Poisson kernel associated to a double quadrature domain
$\Om$ is given by
$$p(z,w)=
Q_0(z,w)+\sum_{k=1}^{n-1} \sigma_k(z)Q_k(w),$$
for $z\in\Om$ and $w\in b\Om$, where
$Q_0(z,w)$ is a rational combination of
$w$, $\bar w$, $z$, $\bar z$, $S(z)$, and $\overline{S(z)}$,
$Q_k(w)$ is a rational combination of
$w$ and $\bar w$, and $\sigma_k(z)$ is equal to the function
$h_k(z)+\overline{H_k(z)}+\ln|z-b_k|$ from Theorem~\ref{thmA},
$h_k'$ and $H_k'$ being rational functions of $z$ and $S(z)$.
\end{thm}

If we use formula~(7.5) from \cite{B1a} to express the Poisson kernel
associated to a double quadrature domain instead, we may deduce that
$p(z,w)$ is a rational combination of
$z$, $S(z)$, $\bar z$, $\overline{S(z)}$,
$w$, and $\bar w$, plus a sum of the form
$$\sum_{k=1}^{n-1}(\lambda_k(z)-\omega_k(z))\mu_k(w),$$
where the functions $\mu_k$ are rational functions of $w$ and $\bar w$.

We have derived formulas for derivatives of the Green's function
in quadrature domains.  It is also possible to deduce formulas for
the Green's function itself.  Indeed, if $\Om$ is a bounded area
quadrature domain without cusps in the boundary, then the Green's
function associated to $\Om$ for a point $b_0\in\Om$ is given by
$$G(z,b_0)=-\ln|z-b_0| + u$$
where $u$ solves the Dirichlet problem on $\Om$ with boundary
data $\ln|z-b_0|$.  We may express
$u$ as we did in the proof of Theorem~\ref{thmA} to write
$$u = h(z)+\overline{H(z)} + \sum_{j=1}^{n-1}c_j\ln|z-b_j|,$$
where, writing $\phi=\ln|z-b_0|- \sum_{j=1}^{n-1}c_j\ln|z-b_j|$, we have
$$h=\frac{P(S_a\phi)}{S_a}\quad\text{and}\quad
H=\frac{P(L_a\overline{\phi}\,)}{L_a}.$$
Let ${\mathcal L}_j(z)=\ln|z-b_j|$ for $j=0,1,2,\dots,n-1$.
We now consider $h$ as a linear combination of
$$h_0=\frac{P(S_a{\mathcal L}_0)}{S_a}$$
and the functions
$$h_j=-\frac{P(S_a{\mathcal L}_j)}{S_a}.$$
The same argument used above in the proof of Theorem~\ref{thmA} yields that
$$h_0=\frac{P(S_a{\mathcal L}_0)}{S_a}=
{\mathcal L}_0 - \frac{\overline{P(L_a\overline{{\mathcal L}_0})}}{\overline{L_a}}.$$
Hence,
$$h_0=
{\mathcal L}_0 + \overline{g},$$
where $g$ is holomorphic on $\Om$.
We may now differentiate this formula along the boundary to obtain
$$h_0'(z)T(z)=\frac{1}{2}\frac{T(z)}{z-b_0}+
\frac{1}{2}\frac{\overline{T(z)}}{\bar z-\bar b_0}
+ \overline{g'(z)}\,\overline{T(z)}$$
for $z$ in the boundary.  This shows that
$$\left(h_0'(z)-\frac{1}{2}\frac{1}{z-b_0}\right)T(z) =
\left(\overline{g'(z)} +\frac{1}{2}\frac{\overline{1}}{\bar z-\bar b_0}\right)
\overline{T(z)}$$
on the boundary.  We can now use the argument that we used in the proof
of Theorem~\ref{thmA} to see that $h_0'$ extends to the double as a meromorphic
function.  Similar reasoning can be applied to $H$.  The bottom line is that
$$G(z,b_0)=-\ln|z-b_0| + h_0(z)+\overline{H_0(z)}+
\sum_{j=1}^{n-1}c_j\left(h_j(z)+\overline{H_j(z)}+\ln|z-b_j|\right),$$
where the $h_j$ and $H_j$ are such that $h_j'$ and $H_j'$ are
algebraic functions of $z$ which are rational combinations of $z$ and
$S(z)$ for $j=0,1,2,\dots,n-1$.  It would be interesting to figure out
the dependence of the functions on the right hand side of this formula
in $b_0$, but the results of \cite{B6a} lead one to believe that it
might be rather messy.

We close this section by remarking that the harmonic
measure functions $\omega_k$ associated to a bounded
area quadrature domain $\Om$ without cusps in the boundary
can be expressed in terms of simpler functions.
There exist real constants $c_j$ so that
$$u:=\omega_k-\sum_{j=1}^{n-1}c_j\ln|z-b_j|$$
is equal to the real part of a holomorphic function $G$
on $\Om$ (i.e., so that the periods vanish).
Note that $G'=2\dee u/\dee z$ is equal to $F_k'$ plus
a linear combination of $1/(z-b_j)$.
Since $F_k'$ and all the $1/(z-b_j)$ extend meromorphically
to the double on area quadrature domains, it follows that
$G'$ extends meromorphically to the double.  Hence, we
may state the following theorem.

\begin{thm}
\label{thmomega}
The harmonic measure functions associated to a bounded area
quadrature domain without cusps in the boundary can be expressed
via
$$\omega_k= \text{\rm Re\ }G_k+\sum_{j=1}^{n-1}c_{kj}\ln|z-b_j|$$
where $G_k$ are holomorphic functions on $\Om$
such that $G_k'$ extends meromorphically to the double.  Consequently,
$G_k'$ are algebraic functions that are rational
combinations of $z$ and the Schwarz function.  The points
$b_j$ are points in the complement of $\Om$, one in the
interior of each bounded component of the complement, and
$[c_{kj}]$ is a non-singular matrix of real constants.
\end{thm}

\section{Mappings between quadrature domains}
\label{secM}

Bj\"orn Gustafsson \cite{G} proved that if $\Om_1$ is a bounded
finitely connected domain bounded by Jordan curves and
$f:\Om_1\to\Om_2$ is a biholomorphic mapping, then $\Om_2$
is an area quadrature domain if and only if $f$ extends
meromorphically to the double of $\Om_1$.
If $f$ is a biholomorphic mapping between two area
quadrature domains, this result can be applied in both
directions.  We also know from \cite{G} that finitely connected
area quadrature domains have Schwarz functions that
extend meromorphically to the domain, and the field of
meromorphic functions on the double is generated by the
extensions of $z$ and the Schwarz function.  Hence, it
follows that $f$ is a rational combination of $z$ and the
Schwarz function $S_1(z)$ for $\Om_1$.  Since $S_1(z)=\bar z$
on the boundary, it follows that $f$ is a rational
function of $z$ and $\bar z$ when restricted to the boundary.
Since $S_1(z)$ is algebraic, $f$ is algebraic.  It is shown in \cite{B5}
(Theorem~1.4) that, on an area quadrature domain, if $H$ extends
to the double, then so does $H'$.  Hence, $f'$ is also a rational
function of $z$ and $S_1(z)$ on $\Om$, and a rational function of
$z$ and $\bar z$ when restricted to the boundary.  In fact, all
the derivatives of $f$ have this property.

We may conclude from this discussion that biholomorphic
mappings between area quadrature domains are rather like
automorphisms of the unit disc (where $S(z)=1/z$).

Gustafsson's result about biholomorphic maps was generalized
to proper holomorphic mappings in \cite{B7} (Theorems~ 1.3 and
p.~168).  Hence, exactly the same conclusions can be drawn about
proper holomorphic mappings between bounded area quadrature domains.
Proper holomorphic mappings between double quadrature domains
are rather like finite Blaschke products.

We next consider a biholomorphic mapping $f:\Om_1\to\Om_2$
between bounded double quadrature domains.  Such a map would have
all of the properties mentioned above, plus the property proved in
\cite{BGS} that $\sqrt{f'}$ belongs to the Szeg\H o span of
$\Om_1$.  Since the Szeg\H o kernel $S(z,a)$ and its derivatives
$(\dee/\dee\bar a)^mS(z,a)$ extend meromorphically
to the double in $z$ for each fixed $a$ on double quadrature domains
(see \cite{BGS}), it follows that $\sqrt{f'}$  is a rational function
of $z$ and $S_1(z)$.  Consequently, $f'$ is the square of such a
function and $f'$ is the square of a rational function of $z$ and
$\bar z$ when restricted to the boundary.

\section{Classes of domains with simple kernels}
\label{secC}

The results of section~\ref{secD} can be pulled back via
conformal mappings.  When the conformal mappings have additional
properties, more detailed information can be gleaned.

Call a bounded finitely connected domain $\Om$ an {\it algebraic\/}
domain if there is a complex algebraic proper holomorphic mapping of $\Om$
onto the unit disc.  This class of domains was studied in \cite{B1a}.
There, it is shown that these domains are characterized by having
algebraic Bergman kernels, or Szeg\H o kernels, or Ahlfors maps.
A list of equivalent ways to characterize the domains is given in
\cite{B1a} (see also \cite{B4}).

Call a bounded finitely connected domain $\Om$ a {\it Briot-Bouquet\/}
domain if there is a proper holomorphic mapping $f$ of $\Om$
onto the unit disc such that $f'$ and $f$ are algebraically dependent,
i.e., such that a polynomial of two complex variables $P(z,w)$ exists
such that $P(f',f)\equiv0$ on $\Om$.  It was shown in \cite{B4} that
domains in this class are characterized by having Bergman
kernels or Szeg\H o kernels generated by only {\it two\/} holomorphic
functions of one complex variable.  A list of equivalent ways to
characterize the domains is given in \cite{B4}.  One of the equivalent
conditions is that Ahlfors maps associated to $\Om$ are Briot-Bouquet
functions.

The next theorem will allow us to harvest the consequences of the
theorems of \S\ref{secD} for algebraic and Briot-Bouquet domains.

\begin{thm}
\label{thmbb}
Suppose $f:\Om_1\to\Om_2$ is a biholomorphic mapping between
bounded domains bounded by finitely many non-intersecting $C^\infty$
smooth curves.  Suppose further that $\Om_2$ is an area quadrature
domain.  Let $S_2(z)$ denote the Schwarz function for $\Om_2$.

If $\Om_1$ is an algebraic domain, then $f$ is algebraic.
Consequently, so is $S_2\circ f$.

If $\Om_1$ is a Briot-Bouquet domain, then $f$ is Briot-Bouquet,
and so is $G\circ f$ whenever $G$ is a meromorphic function on
$\Om_2$ that extends meromorphically to the double of $\Om_2$.
In particular, $S_2\circ f$ is Briot-Bouquet.

In both cases, $f$ and $S_2\circ f$ extend meromorphically to
the double of $\Om_1$ and form a primitive pair for the field
of meromorphic functions on the double.
\end{thm}

\begin{proof}[Proof of Theorem~\ref{thmbb}]
It is shown in \cite{B5} that area quadrature domains are algebraic
domains, and it is shown in \cite{B6} that biholomorphic mappings
between algebraic domains are algebraic (since biholomorphic maps
in Bergman coordinates are algebraic and Bergman coordinates themselves
are rational combinations of the Bergman kernel, which is algebraic
in algebraic domains).  Hence,the first part of the theorem is proved.

Suppose now that $\Om_1$ is a Briot-Bouquet domain.  Since $\Om_2$
is an area quadrature domain, the mapping $f$ extends to the double
of $\Om_1$ as a meromorphic function.  Assume that $G$ extends
meromorphically to the double of $\Om_2$.  Since the property of
extending to the double is preserved under conformal changes of
variables, it follows that $G\circ f$ also extends meromorphically
to the double of $\Om_1$.

It is shown in \cite{B1b} that the field of meromorphic functions
on the double of a smoothly bounded finitely connected domain is
generated by two Ahlfors mappings associated to the domain.  Hence,
there are Ahlfors maps $g_1$ and $g_2$ associated to two distinct
points in $\Om_1$ such that $f(z)=R(g_1(z),g_2(z))$, where $R$ is
a complex rational function of two variables.  Differentiating
this formula reveals that $f'$ is equal to
$g_1'$ times a rational function of $g_1$ and $g_2$ plus
$g_2'$ times a rational function of $g_1$ and $g_2$.
Since $\Om$ is Briot-Bouquet, the Ahlfors maps are Briot-Bouquet
functions (see \cite{B4}).  Hence, there are polynomials $P_1$ and
$P_2$ of two complex variables such that $P_j(g_j',g_j)\equiv0$, $j=1,2$.
Now the procedure used in \cite{B4} to construct the compact
Riemann surface to which all the kernel functions extend can
be applied here.  Since $g_1$ and $g_2$ extend meromorphically to
the double of $\Om_1$, we may use the identities
$P_j(g_j',g_j)\equiv0$, $j=1,2$,
to extend $g_1'$ and $g_2'$ as meromorphic functions on a compact
Riemann surface that is a finite branched cover of the double of
$\Om_1$.  Now both $f$ and $f'$ extend meromorphically to this
compact Riemann surface, and they are therefore algebraically
dependent.  Hence, there is a polynomial $P$ of two complex variables
such that $P(f',f)\equiv0$ on $\Om$, i.e., $f$ is a Briot-Bouquet
function.  The same reasoning can be applied to $G\circ f$ to
conclude that it, too, is a Briot-Bouquet function.
\end{proof}

If $f:\Om_1\to\Om_2$ is a biholomorphic mapping between smoothly
bounded domains, then the Green's functions associated to the
domains transform via
$$G_1(z,w)=G_2(f(z),f(w)).$$
Consequently,
$$G_1^{(1)}(z,w)=
G_2^{(1)}(f(z),f(w))f'(w),$$
and we can use the results of \S\ref{secD} to pull back
results about the Poisson kernel and Green's functions on double
quadrature domains to algebraic and Briot-Bouquet domains.
Before we state the resulting theorem, note that the
complex unit tangent function transforms via
$$T_2(f(z))=\frac{f'(z)}{|f'(z)|}T_1(z),$$
and so
$$p_1(z,w)=p_2(f(z),f(w))|f'(z)|$$
(which could also be deduced directly from considerations
of arc length, but the transformation rule for the complex
unit tangent vector functions is worth writing down here
because it allows one to pull back complexity results about
the tangent function from double quadrature domains).

\begin{thm}
\label{thm4}
Suppose that $\Om$ is a bounded domain bounded by finitely
many non-intersecting $C^\infty$ smooth curves that is either
an algebraic or Briot-Bouquet domain, and suppose that
$f(z)$ is a biholomorphic mapping of $\Om$ onto a double
quadrature domain with Schwarz function $S(z)$.
The Poisson kernel associated to
$\Om$ is given by
$$p(z,w)=
Q_0(z,w)+\sum_{k=1}^{n-1} \sigma_k(z)Q_k(w),$$
for $z\in\Om$ and $w\in b\Om$, where
$Q_0(z,w)$ is a rational combination of
$f(w)$,
$S(f(w))$,
$f(z)$,
$S(f(z))$,
$\overline{f(z)}$, and
$\overline{S(f(z))}$ times $|f'(w)|$,
$Q_k(w)$ is a rational combination of
$f(w)$ and $S(f(w))$ times $|f'(w)|$,
and $\sigma_k(z)$ is equal to
$(g_k(z)+\overline{G_k(z)}+\ln|f(z)-b_k|)$,
where
$g_k$ and $G_k$ have derivatives that are
rational functions of $f(z)$ and $S(f(z))$ times $f'(z)$.
In the case that $\Om$ is an algebraic domain, $f(z)$, $f'(z)$,
and $S(f(z))$ are algebraic and $|f'(z)|$ is real algebraic,
and
$g_k$ and $G_k$ have algebraic derivatives.
In the case that $\Om$ is a Briot-Bouquet domain, $f(z)$,
$S(f(z))$, and $f'(z)$ are Briot-Bouquet functions, and
$g_k$ and $G_k$ have derivatives that are
Briot-Bouquet functions times $f'(z)$.
\end{thm}

In Theorem~\ref{thm4}, since an algebraic domain is also a
Briot-Bouquet domain, the techniques of \cite{B6} can be used
to show that $f'(z)$ extends meromorphically to a compact
Riemann surface that is a finite cover of the double to which
$f(z)$ and $S(f(z))$ also extend.  Hence, if $G_1$ and $G_2$ are
a primitive pair for the compact Riemann surface, all of the
functions $f$, $f'$, and $S\circ f$ are rational combinations
of the same two functions $G_1$ and $G_2$.  Hence the Poisson
kernel is composed of basic building blocks that are surprisingly
simple.

\section{Classical boundary operators on double quadrature domains}
\label{secO}

Suppose that $\Om$ is a bounded double quadrature domain.  We will
now show that the Szeg\H o projection associated to $\Om$, as an
operator from $L^2(b\Om)$ to itself, maps rational functions of $z$
and $\bar z$ to the same class of functions.  Suppose that
$u(z)=R(z,\bar z)$ is such a function without singularities on
the boundary $b\Om$.
The Szeg\H o projection $P$ satisfies
$$Pv=v-\overline{P(\,\overline{vT}\,)}\overline{T}$$
on the boundary (see p.~13 of \cite{B1}).  Now,
since $u$ extends antimeromorphically to $\Om$
as $R(\,\overline{S(z)},\bar z)$ and since $T(z)$ is equal
to the boundary values of a meromorphic function on the double
without poles on the boundary curves, we see that the holomorphic
function $Pu$ has the same boundary values as an antimeromorphic
function on $\Om$ that extends smoothly up to the boundary.  Hence,
$Pu$ extends meromorphically to the double, and is therefore a
rational combination of $z$ and $S(z)$.  Now, restricting to the
boundary shows that $Pu$ is a rational combination of $z$ and
$\bar z$, and the proof is complete.  It also follows that the
extension of $Pu$ to $\Om$ as a holomorphic function is a rational
function of $z$ and $S(z)$ without singularities on $\Obar$.
Hence, as an operator from $L^2(b\Om)$ to the space of holomorphic
functions on $\Om$ with $L^2$ boundary values, $P$ maps rational
functions of $z$ and $\bar z$ to algebraic functions.

It is easy to see that the Cauchy transform has the same behavior
on an area quadrature domain.  Indeed, we may write
$$({\mathcal C}u)(z)=\frac{1}{2\pi i}\int_{b\Om} \frac{u(w)}{w-z}\ dw
=\frac{1}{2\pi i}\int_{b\Om} \frac{R(w,S(w))}{w-z}\ dw,$$
and the Residue Theorem yields that $Cu$ is a rational function of
$z$ with poles at the poles of the meromorphic function $R(w,S(w))$
in $\Om$.  The $L^2$ adjoint of the Cauchy transform satisfies
$${\mathcal C}^* u = u - \overline{{\mathcal C}(\,\overline{uT}\,)T}$$
(see \cite{B1}, p.~12).  If $\Om$ is a double quadrature domain,
then $T$ is the restriction to the boundary of a rational function
of $z$ and $\bar z$.  Hence, it follows that ${\mathcal C}^*$
maps the space of rational functions of $z$ and $\bar z$ without
singularities on the boundary into itself.  It now follows that
the Kerzman-Stein operator ${\mathcal A}={\mathcal C}-{\mathcal C}^*$
enjoys the same property.

We close this section by remarking that, since
$dz=T(z)ds$ and $d\bar z = \overline{T(z)}ds$,
and since $T(z)$ extends to the double as a meromorphic
function or as an antimeromorphic function on a double quadrature
domain $\Om$, it is possible to convert an integral of the form
$$\int_{b\Om} R(z,\bar z)\ ds$$
where $R(z,\bar z)$ is a rational function of $z$ and $\bar z$
to an integral
$$\int_{b\Om} R(z,S(z))G(z)\ dz$$
where $G(z)$ and the Schwarz function $S(z)$ are meromorphic
functions, and consequently, the integral can be computed by
means of the residue theorem.  The same conversion can be made
for integrals involving $dz$ or $d\bar z$ in place of $ds$.
Thus, computing integrals in double quadrature domains is
rather like computing integrals of rational functions of
trigonometric functions on the unit circle.

\section{Fourier Analysis on multiply connected domains}
\label{secF}
The results of this paper are probably most interesting
in the realm of double quadrature domains, however, the
ideas espoused here can be viewed as a new way of thinking
about Fourier analysis in multiply connected domains.
Classical Fourier Analysis is expanding functions in
powers of $e^{i\theta}$ and $e^{-i\theta}$.  If we
write $e^{i\theta}=z$, then $e^{-i\theta}=\bar z = 1/z$, and
$1/z$ is the Schwarz function for the unit disc.  On the unit
disc, the Poisson extension of a rational function $R(z,\bar z)$
can be gotten by first extending $R(z,\bar z)$ as the meromorphic
function $R(z,S(z))$ and then subtracting off appropriate
derivatives $(\dee/\dee w)^m G(z,w)$ of the Green's function
at points $w$ where the meromorphic function has poles.
Since the Green's function is
$$-\ln\left|\frac{z-w}{1-\bar w z}\right|,$$
these derivatives are rational functions of $z$ and $\bar z$.
The extension obtained in this way is the same as the harmonic
extension of the Fourier series from the boundary obtained
in Classical Analysis.  (This gives another way to see that
the harmonic extension of a rational function on the boundary
of the unit disc is also rational.  See \cite{E} and
\cite{BEKS} for more on this subject in multiply connected
domains.)

In this paper, we have generalized this idea to the
multiply connected setting.  If $\Om$ is a bounded finitely
connected domain bounded by $n$ non-intersecting $C^\infty$
smooth curves, then it is proved in \cite{B1b} that
there are two Ahlfors maps $g_1$ and $g_2$ associated to
$\Om$ such that the field of meromorphic functions on
the double of $\Om$ is generated by $g_1$ and $g_2$.
Note that $g_1$ and $g_2$ are proper holomorphic maps
of $\Om$ onto the unit disc, and that $|g_j(z)|=1$
when $z\in b\Om$, $j=1,2$.  These functions will take
the place of $e^{\pm i\theta}$ in what follows.

The rational functions of $z$ and $\bar z$ on the unit
circle are exactly the functions that extend from the
unit circle to the double of the unit disc as meromorphic
functions.  On $\Om$, this class corresponds to the class
of all rational functions of
$g_1(z)$,
$g_2(z)$,
$\overline{g_1(z)}$, and
$\overline{g_2(z)}$, which is the same as the class of
all rational functions of $g_1(z)$ and $g_2(z)$, since
$\overline{g_j(z)}=1/g_j(z)$ when $z\in b\Om$, $j=1,2$.
It is interesting to note that this class of functions is
dense in the space of $C^\infty$ functions on the boundary,
and, hence, also dense in the space of continuous functions
on the boundary.  Indeed, any $C^\infty$ function on the
boundary can be decomposed as $h+\overline{HT}$ where $h$
and $H$ are holomorphic functions in $C^\infty(\Obar)$
(see \cite{B1}, p.~13).  It was proved in \cite{B0}
(see also \cite{B1}, p.~29) that the complex linear span
of functions of $z$ of the form $S(z,b)$ as $b$ ranges
over $\Om$ is dense in the space of holomorphic functions
in $C^\infty(\Obar)$.  Fix a point $a\in\Om$, and let
$S_a(z)=S(z,a)$ and $L_a(z)=L(z,a)$ as always.
The Szeg\H o and Garabedian kernels satisfy the identity
\begin{equation}
\label{SL}
\overline{S_a}=\frac{1}{i}L_aT
\end{equation}
on the boundary (see \cite{B1}, p.~24).
Given a
function $\varphi$ in $C^\infty(b\Om)$, we may decompose
$S_a\varphi$ as $h+\overline{HT}$ and we may use
identity~(\ref{SL}) to obtain
$$\varphi=\frac{h}{S_a}-i\frac{\overline{H}}{\overline{L_a}}.$$
Next, if we approximate $h$ and $H$ by linear combinations
of functions of the form $S_b$, we see that $\varphi$ can be
approximated by linear combinations of $S_b/S_a$ and the
conjugates of $S_b/L_a$.  But such quotients extend
meromorphically to the double of $\Om$ (since
identity~(\ref{SL}) reveals that $S_b/S_a$ is equal to the
conjugate of $L_b/L_a$ on the boundary, and $L_b/S_a$ is
equal to the conjugate of $S_b/L_a$ on the boundary).
Hence, the space of functions on the boundary that are
restrictions of meromorphic functions on the double without
poles on the boundary is dense in $C^\infty(b\Om)$.

Consequently, one approach to solving the Dirichlet
problem would be to first approximate a continuous
function on the boundary by rational combinations of
$g_1$ and $g_2$.  The approximation can be extended
meromorphically to the double.  Next, the poles of the
extended function can be subtracted off by derivatives
of the Green's function in the second variable to obtain
a solution to the Dirichlet problem for the approximation.
Various forms of the derivatives of the Green's function
can be used to deduce facts about the complexity of
the solution.  For example, Theorem~1.2 of \cite{B6a}
expresses the Green's function in finite terms involving
a single Ahlfors map, say $g_1$.  Or Theorem~\ref{thmB}
of the present work can be applied to see that the
solution is a rational function of $g_1$ and $g_2$
and their conjugates modulo an explicit $n-1$ dimensional
subspace.

\end{document}